\documentclass[11pt]{amsart}
\usepackage{latexsym,amssymb,amsmath,hyperref,empheq}
\usepackage{amsthm, dsfont}
\usepackage{tikz-cd}
\usepackage[all]{xy}
\usepackage[short,nodayofweek]{datetime}
\usepackage[active]{srcltx}

\leftmargin  \leftmargini
\def\ZZ{{\mathbb Z}}

\def\RR{{\mathbb R}}

\def\C{{\mathbb C}}
\def\F{{\mathbb F}}

\def\Z{{\mathbb Z}}
\def\Q{{\mathbb Q}}

\def\D{{\mathbb D}}
\def\G{{\mathbb G}}
\def\U{{\mathbb U}}
\def\R{{\mathbb R}}

\def\pp{{\mathfrak p}}

\def\n{\operatorname N}

\DeclareMathOperator{\GL}{GL}

\DeclareMathOperator{\trace}{tr}

\DeclareMathOperator{\SL}{SL}

\newcommand{\ring}{\mathfrak o}
\newcommand{\p}{\mathfrak p}
\newcommand{\SU}{\mathop{SU}}

\theoremstyle{plain}
\newtheorem{thm}{Theorem}[section]
\newtheorem{cor}[thm]{Corollary}
\newtheorem{lem}[thm]{Lemma}

\newtheorem{prop}[thm]{Proposition}

\theoremstyle{definition}
\newtheorem{defn}[thm]{Definition}
\newtheorem{example}[thm]{Example}

\theoremstyle{remark}

\newtheoremstyle{Acknowledgements}
  {}
    {}
     {}
     {}
    {\bfseries}
    {}
     {.5em}
     {\thmname{#1}\thmnumber{ }\thmnote{ (#3)}}
\theoremstyle{Acknowledgements}
\newtheorem{ack}{Acknowledgements.}

\definecolor{red}{rgb}{1,0,0}

\definecolor{orange}{rgb}{1,0.5,0}

\definecolor{purple}{rgb}{.5,.2,.8}

\definecolor{blue}{rgb}{.2,.2,.8}

\definecolor{green}{rgb}{.4,.6,.4}

\definecolor{brown}{rgb}{.6,.3,.3}

\date{\today, \currenttime} 

\begin{document}
\title[Unitary groups associated to division algebras]{On unitary groups associated to division algebras of degree three}

\author{Kathrin Maurischat }
\address{\rm {\bf Kathrin Maurischat}, Mathematisches Institut,
   Heidelberg University, Im Neuenheimer Feld 205, 69120 Heidelberg, Germany }
\curraddr{}
\email{\sf maurischat@mathi.uni-heidelberg.de}


\begin{abstract}
We show that the special unitary group associated to an involution of the second kind on a central division algebra of degree three does not contain hermitian or skew-hermitian elements.
Especially, there are no reflections.

For Albert's special cyclic presentation we show that the intersections of reasonable $S$-arithmetic subgroups with the two obvious maximal subfields consist of the trivial central elements.
\end{abstract}

\maketitle
\setcounter{tocdepth}{2}
\tableofcontents
\section{Introduction}
The algebraic and arithmetic structure of central semisimple algebras was developed in the 1920s and 1930s by Albert, Brauer, Deuring, Dickson, Hasse, E.~Noether, and others.
There exist  broad theories for them, like local-global principles, cohomology theories, or Brauer-Severi varieties.
In particular, Albert studied division algebras with involution for his work on Riemann manifolds.
However, division algebras stick to be impervious objects, and we may say that  so far only the case of quaternion algebras is understood comprehensively.

Here we present some results on  unitary groups which arise from involutions of the second kind on  division algebras of degree three over a number field.
Those are inner forms of special unitary groups $\SU_3(H)$ of matrices defined by a hermitian matrix $H\in\GL_3$.
Albert  \cite{albert2} showed that simple division algebras $D$ of odd prime degree over number fields carry an involution $\alpha$ of the second kind 
if they have a maximal subfield which is cyclic Galois not only over the center but over its subfield fixed by the involution, and that these algebras are cyclic.
We give an exact statement of this result for degree three in Theorem~\ref{thm-involutions}, and this special cyclic presentation will be referred to
throughout this paper.
The unitary group $\mathop U$ is given by those $g\in D$ such that $\alpha(g)g=1$, 
and the special unitary group consists of those elements which additionally are of reduced norm one, $\n_{rd}(g)=1$.

Our first main result is the following. It is surprising as it contradicts our intuition that, apart from rotations, reflections are exemplary elements of  special unitary groups.
\begin{thm}[Corollary~\ref{no-elements-of-even-order}]
  The special unitary group $\SU$ of a division algebra of degree three with involution of the second kind does not contain non-trivial elements of even order.
 Especially, there are no reflections in $\SU$.
\end{thm}
This is proved by the more general observation that the intersections of the $\alpha$-eigenspaces with $\SU$ are trivial, i.e. there aren't any hermitian or skew-hermitian elements in $\SU$.
The theorem has a far-reaching consequence.
\begin{cor}[Corollary~\ref{reflections-non-division}]
 Let $H\in\GL_3(\C)$ be a hermitian matrix, and let $\SU_3(H)=\{g\in\SL_3(\C)\mid \bar g'Hg=H\}$ denote the associated special unitary group.
 Let $R\in\SU_3(H)$ be a reflection. Then $R$ is not contained in any  global division algebra
 $(D,\alpha)$ of degree three with involution $\alpha$ of the second kind  such that for some infinite place $v$,  $\SU(D_v)\cong \SU_3(H)$.
 Similarly, no hermitian or skew-hermitian element of $\SU_3(H)$ arises in this way.
\end{cor}

The second section of our results is concerned with  $S$-arithmetic subgroups of the special unitary group. For this, the structure of the algebra $D$ and its involution $\alpha$ 
must allow an integral model.
Arithmetic subgroups provide  prototypes of discrete subgroups of the local groups $\SU(F_\pp)$, which under mild assumptions,
such as the compactness of $\SU(F_v)$ at some infinite place $v$ (\cite{borel}, \cite{borel-harder}),  are seen to be cocompact.
While they arise at many points in modern mathematics, 
there is no explicitly computed example of a discrete cocompact global subgroup for the special unitary groups in question so far.
We use the special cyclic presentation of $D$ mentioned above,
\begin{equation*}
 D\:=\: L\oplus Lz\oplus Lz^2\:,
 \end{equation*}
where $L$ is $C_6$-Galois over the totally real field $F$, the center of $D$ being an imaginary quadratic extension $E$ of $F$,
and assume the cyclic presentation as well as the involution is defined over $\ring_F$. Then the special unitary group $\SU$ gives rise to a group scheme $\G$ defined over the integers $\ring_F$.
A criterion for this is given by Proposition~\ref{maximal-orders}.
Then the main result of Section~\ref{section-integers} may be formulated as follows.
\begin{thm}[Proposition~\ref{prop-S-monomials}, Theorem~\ref{thm-no-o_E(S)-points}]
 Let $S$ be a set of prime ideals $\pp$ not dividing $2$, which are inert in $E$ but split in $L$, and such that $F_\pp$ does not contain the third roots of unity.
 Then the intersections of the $S$-arithmetic subgroup $\G(\ring_F(S))$ with the maximal subfields $L$ and $E(z)$ both are given by the third roots of unity contained in $E$.
\end{thm}
Prime ideals $\pp$ satisfying the assumption of the theorem lead to proper special unitary groups $\SU_3(E_\pp)$, i.e. non-split and non-isomorphic to $\SL_3(E_\pp)$.
The assumption on $F_\pp$ not containing the sixth roots of unity may seem artificial at a first glance. But if $E$ is chosen to be the Kummer extension $F(\zeta_3)$,
 this is satisfied for all inert prime places.
The theorem tells us that non-trivial elements of the promising $S$-arithmetic subgroups $\G(\ring_F(S))$ belong to non-obvious splitting fields of $D$.
The results of Section~\ref{section-integers} are more detailed, for example they give a criterion for the denominators of the $\mathop{Gal}(L/E)$-fixed points of $\G(F)$ 
(Corollary~\ref{denominators}).
\bigskip

The paper is organized as follows.
In Section~\ref{section-involutions} we recall Albert's theorem and the resulting special cyclic presentation. 
In Section~\ref{section-unitary-groups} we define the unitary groups and discuss the problem of defining them as integral group schemes.
A criterion for definiteness of the hermitian forms involved is given in Section~\ref{section-definitness}.
The results on the order of elements are found in Section~\ref{section-fixed-points}, where we provide a discussion of fixed points of the unitary groups for several actions
on their coefficients. 
Finally, in Section~\ref{section-integers} we give the properties of $S$-arithmetic subgroups. 
Here we restrict to the case that the center $E$ is  imaginary quadratic over the totally real ground field $F$ fixed by the involution.
This means that at the infinite places we indeed get unitary groups (instead of copies of $\GL_3$), which in view of the applications we have in mind is sensible.
The results can be obtained in greater generality as results on the coefficients with respect to the special cyclic presentation with moderate requirements on the involution. 
Also, most of them are true for the unitary group as well.
So we present and prove them in general and give the results on $S$-arithmetic subgroups of special unitary groups afterwards as Example~\ref{example-S-arithmetic}.
The used methods are of astonishing simple algebraic and arithmetic nature.
\bigskip

\begin{ack}
The author  thanks Cristina Ballantine and Brooke Feigon for their encouragement to publish the results in hand separately from the joint work.
Also, thanks go to Andreas Maurischat for the help  from his bag of algebraic tricks.
\end{ack}

\section{Preliminaries}\label{preliminaries}
\subsection{Involutions of the second kind}\label{section-involutions}
Let  $E/F$ be an extension of number fields of degree two, and denote by $\tau$ the non-trivial Galois automorphism.
We often abbreviate $\tau$ by  conjugation, $\tau(x)=\bar x$.
Let $D$ be a central simple division algebra over $E$ of degree three.
An involution of the second kind on $D$ is an anti-automorphism $\alpha$ such that when restricted to the center $E$ it equals $\tau$, $\alpha\mid_E=\tau$.
Throughout this paper we will use the following cyclic presentation of $D$. In this it is crucial that there exists a field extension $L/E$ contained in $D$ such that $L/F$ is a 
Galois extension with cyclic Galois group $\mathop{Gal}(L/F)\cong C_6$. 
This is a result due to Albert~\cite{albert2}, especially Theorem~22, and valid in much more generality.
We denote by $\rho$ a generator of the Galois group $\mathop{Gal}(L/E)$, and extend $\tau$ to $L$. Then
$\mathop{Gal}(L/F)=<\rho\tau>$. 
\begin{thm}\label{thm-involutions}[See~\cite{albert2}]
Let $E/F$ be a quadratic extension of number fields and denote by $\tau$ its non-trivial Galois automorphism.
 Let $D$ be a central simple division algebra over $E$ of degree three. 
 Then $D$ carries an involution of the second kind extending $\tau$ if and only if the following two conditions are satisfied.
 \begin{itemize}
  \item [(i)]
  There exists a maximal subfield $L$ in $D$ such that $L/F$ is $C_6$-Galois. 
  Consequently, a cyclic realization of $D$ is given by
  $D=L\oplus Lz\oplus Lz^2$, subject to the relations $z^3=a\in E^\times$ and $zl=\rho(l)z$ for all $l\in L$, where $\rho$ denotes a non-trivial element of $\mathop{Gal}(L/E)$.
  \item [(ii)]
  There exists a solution $b\in M=L^\tau$  of the norm equation
  \begin{equation*}
 \n_{E/F}(a)\:=\:\n_{M/F}(b)\:.
\end{equation*}
 \end{itemize}
In this case, an involution $\alpha$ of the second kind is given by the obstructions
\begin{equation*}
 \alpha\mid_L\:=\:\tau\quad \textrm{ and }\quad \alpha(z)\:=\:\frac{b}{a}z^2\:.
\end{equation*}
Any other involution $\beta$ of the second kind is conjugate to $\alpha$ by some element $c\in M^\times$,
\begin{equation*}
 \beta(d)\:=\: c^{-1}\alpha(d)c\:,
\end{equation*}
for all $d\in D$, and in particular $\beta(z)=c^{-1}\rho^2(c)\alpha(z)$.
Moreover, a cyclic algebra as described in (i) is a division algebra if and only if $a$ is not a norm of $L$, $a\not\in \n_{L/E}$. 
\end{thm}
For the field extension $L/F$ we have  the following picture.

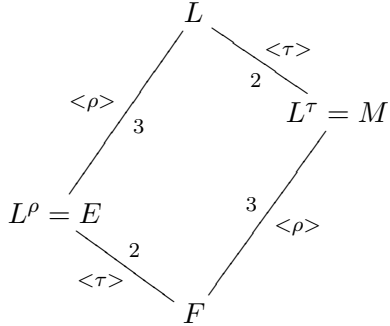
\begin{xy}
 \xymatrix{
 &L&\\
 &&L^\tau=M \ar@{-}[lu]^2_{<\tau>}\\
 L^\rho=E \ar@{-}[ruu]^{<\rho>}_3&&\\
 &F \ar@{-}[ruu]^3_{<\rho>} \ar@{-}[lu]^{<\tau>}_2& 
 }
\end{xy}
Notice that the property $a\notin\n_{L/E}$ together with $a\bar a=\n_{M/F}(b)$ force $a\in E\setminus F$, as otherwise $a^2=\n_{M/F}(b)\in\n_{L/E}$. 
But then $a=\n_{M/F}(b^2a^{-1})$ belongs to $\n_{L/E}$, too.
The structure constant $a$ of $D$ is unique up to a factor $\n_{L/E}(l)$ for some $l\in L$. So it is always possible to choose  $a\in\ring_E$, the ring of integers of $E$.
Throughout this paper, we refer to a division algebra $D$ with constants $a\in\ring_E\setminus \ring_F$ and $b\in M$ satisfying the special cyclic presentation of Theorem~\ref{thm-involutions}.

We have the following embedding of the cyclic algebra $D=L\oplus Lz\oplus Lz^2$ to the matrix ring $M_3(L)$ defined by
\begin{equation*}
 L\ni l\:\mapsto\: \begin{pmatrix}l&&\\&\rho(l)&\\&&\rho^2(l) \end{pmatrix}\:,\quad
  z\:\mapsto\:\begin{pmatrix}&1&\\&&1\\a&&\end{pmatrix}\:.
\end{equation*}

We will often identify $g=l_0+l_1z+l_2z^2$, $l_j\in L$, with its image
\begin{equation}\label{embedding}
 g\:=\:g(l_0,l_1,l_2)\:=\:\begin{pmatrix}  l_0&l_1&l_2\\a\rho(l_2)&\rho(l_0)&\rho(l_1)\\a\rho^2(l_1)&a\rho^2(l_2)&\rho^2(l_0)\end{pmatrix}\:.
\end{equation}
Under this embedding, the involution $\alpha$ is realized as
\begin{equation}\label{matrix-involution}
 \alpha(g)\:=\:
 \begin{pmatrix}
  \bar l_0&\frac{\rho(\rho(b)\rho^2(b)\bar l_2)}{a}&\frac{\rho^2(\rho(b)\bar l_1)}{a}\\
  \rho(b)\bar l_1&\rho(\bar l_0)&\frac{\rho^2(\rho(b)\rho^2(b)\bar l_2)}{a}\\
  \rho(b)\rho^2(b)\bar l_2&\rho(\rho(b)\bar l_1)& \rho^2(\bar l_0)
 \end{pmatrix}.
\end{equation}

\subsection{Unitary groups}\label{section-unitary-groups}
Let $D$ be a cyclic algebra satisfying Theorem~\ref{thm-involutions} with structure constant $a\in\ring_E$ and  with involution $\alpha$ attached to a fixed $b\in M$.
The involution $\alpha$ gives rise to a non-degenerate hermitian form $h$ on  $D$,
\begin{eqnarray*}
 h:D\times D &\longrightarrow &D\\
 (x,y)&\mapsto& \alpha(x)y\:.
\end{eqnarray*}
That is, for all $\lambda,\mu,x,y\in D$ we have
\begin{equation*}
 h( x\lambda, y\mu)\:=\: \alpha(\lambda)\alpha(x)y\mu=\alpha(\lambda)h(x,y)\mu
\end{equation*}
as well as 
\begin{equation*}
 \alpha(h(x,y))\:=\: \alpha(y)x\:=\:h(y,x)\:.
\end{equation*}
The unitary group of this hermitian form is 
\begin{equation*}
\mathop U\:=\:\{g\in D^\times\mid h(gx,gy)=h(x,y)\textrm{ for all } x,y\in D\}\:=\: \{g\in D^\times\mid \alpha(g)g=1\}\:,
\end{equation*}
and the special unitary group
\begin{equation*}
 \mathop{SU}\:=\:\{g\in\mathop U\mid \n_{rd}(g)=1\}
\end{equation*}
is the subgroup of those $g$ with reduced norm one, which using the embedding (\ref{embedding}) is given by the determinant.
Let $\ring_D$ be the maximal order of $D$ given by its integral elements.
We define a  unitary group scheme $\U$ over $\ring_F$
forcing its $F$-valued points to be $\U(F)=\mathop U$,
\begin{equation*}
\U(R)\:=\:\{g\in \ring_D\otimes_{\ring_F} R\mid \alpha(g)g=1\}
\end{equation*}
for all extension rings $R$  of $\ring_F$.
Similarly, we have the special unitary group scheme $\G$ such that $\G(F)=\mathop{SU}$,
\begin{equation*}
 \G(R)\:=\:\{g\in \ring_D\otimes_{\ring_F} R\mid \alpha(g)g=1,\: \n_{rd}(g)=1\}\:.
\end{equation*}

We can also define  unitary schemes over $F$ thereby giving in the notion of $\ring_F$-valued points.
For example, it is very appealing to use the embedding of $D$ into $M_3(L)$ as $L$-algebra given by (\ref{embedding}).
But this is defined  over $\ring_L$ if and only if $a$ is a unit in $\ring_E$.
There is an extension  of the involution $\alpha$ on $M_3(L)$,
\begin{equation}\label{matrix-invol}
 \alpha(\begin{pmatrix}
         m_{00}&m_{01}&m_{02}\\
         m_{10}&m_{11}&m_{12}\\
         m_{20}&m_{21}&m_{22}
        \end{pmatrix})
        \:=\:
\begin{pmatrix}
 \bar m_{00}&\rho(b^{-1})\bar m_{10}&\rho(b^{-1})\rho^2(b^{-1})\bar m_{20}\\
  \rho(b)\bar m_{01}&\bar m_{11}&\rho^2(b^{-1})\bar m_{21}\\
  \rho(b)\rho^2(b)\bar m_{02}&\rho^2(b)\bar m_{12}& \bar m_{22}
\end{pmatrix}\:,
\end{equation}
and a group scheme $\D$ defined over $F$ as a subscheme of $M_3(L)$ such that $\D(F)\subset M_3(L)$ coincides with the image of $D$, 
and  unitary group schemes with $F$-valued points $\mathop{U}$, $\mathop{SU}$ can be defined directly as subgroups of $\GL_3(L)$.
We will refer to these schemes as the cyclic presentations associated to Theorem~\ref{thm-involutions}.
As long as we are concerned with $K$-valued points, $K$ an extension of $F$, the two notions  coincide, and we will make frequent use of the cyclic presentation. 
We can at least talk about the set $\Lambda$ of $\ring_L$-points of $\D(F)$ meaning those with matrix entries in $\ring_L$, $\Lambda=\ring_L\oplus \ring_Lz\oplus\ring_Lz^2$.
But notice  that in case $a\notin \ring_E^\times$ or $b\notin \ring_M^\times$ the involution $\alpha$ will neither act on $M_3(\ring_L)$ nor on the $\ring_L$-points $\Lambda$ of $\D(F)$.

In contrary, $\alpha$ acts on $\D(\ring_F)=\ring_D$, as for an integral element $d\in\ring_D$, the minimal polynomial $f$ belongs to $\ring_E[X]$, 
and the polynomial with conjugate coefficients $\bar f\in\ring_E[X]$ is a minimal polynomial for $\alpha(d)$.
But  if and only if  $\ring_D$ coincides with $\Lambda$, then
$\D$ is already defined over $\ring_F$, and thus is a matrix realization of the group scheme given by the maximal order $\ring_D$ of integers in $D$,
and the notions of integral points of $\U$ and $\G$ coincide with the notions of the $\ring_L$-valued points of the corresponding cyclic presentation group.

\begin{prop}\label{maximal-orders}
 The order $\Lambda=\ring_L+\ring_Lz+\ring_Lz^2$ of $D$ equals the maximal order $\ring_D$ of integral elements if and only if $a$ is a unit in $\ring_E$.
\end{prop}
\begin{proof}[Proof of Proposition~\ref{maximal-orders}]
First notice that  $\Lambda$ is contained in $\ring_D$.
  By~\cite[11.6]{reiner}, the order $\Lambda$ is maximal if and only if for all prime ideals $\pp$ of $\ring_E$ the localization $\Lambda_\pp$ is a maximal order in $D_\pp$.
    Assume $a\in\ring_E\setminus \ring_E^\times$ and let $\pp$ be a prime ideal of $\ring_E$ which contains $a$.
  As $b\in M$ is chosen such that $a\bar a=\n_{L/E}(b)$, we find that the element $b^{-1}z^2\in D$ has minimal polynomial $X^3-\frac{a}{\bar a}$.
  As $v_\pp(\frac{a}{\bar a})=0$, we see that $b^{-1}z^2$ is an integer of $D_\pp$. But $b^{-1}\notin \ring_{L_\p}$. 
  So the maximal order  $\ring_{D_\pp}$ is strictly larger than $\Lambda_\pp$.
  
On the other hand, the discriminant of $\Lambda$ is given by  its generator
  \begin{equation*}
   \mathop{disc}(\Lambda)\:=\:\det(\trace\nolimits_{\mathop{rd}}(b_{jk}b_{j'k'}))
  \end{equation*}
for the $\ring_E$-basis $b_{jk}=e_jz^k$, $j,k=0,1,2$, of $\Lambda$. Here $e_j$, $j=0,1,2$, is any $\ring_E$-basis of $\ring_L$, which we may choose $\mathop{Gal}(L/E)$-invariant.
As $\trace_{\mathop{rd}}(lz)=0=\trace_{\mathop{rd}}(lz^2)=0$ for all $l\in L$, we easily compute
\begin{equation*}
  \mathop{disc}(\Lambda)\:=\:\det\begin{pmatrix} \trace(e_je_k)&0&0\\0&0&a\trace(e_j\rho(e_k))\\0&a\trace(e_j\rho^2(e_k))&0                                  
                                         \end{pmatrix}
\:=\:\bigl( -a^2\mathop{disc}(\ring_L)\bigr)^3\:.
\end{equation*}
As $\ring_L$ is the maximal $\ring_E$-order (of integral elements) in $L$, the order $\ring_D$ of integral elements is an $\ring_L$-module, and the different $\mathfrak D$ of $\ring_D$ 
is an $\ring_L$-module, too. It follows that the ideal norm $ \n(\mathfrak D)$ of $\mathfrak D$ is divisible by $\mathop{disc}(\ring_L)$.
But by~\cite[25.10]{reiner},
\begin{equation*}
 \mathop{disc}(\ring_D)\:=\: \bigl( \n(\mathfrak D)\bigr)^3\:,
\end{equation*}
and further, maximal orders belong to minimal discriminants and vice versa.
So in case $a\in\ring_E^\times$ is a unit, we see that $\mathop{disc}(\Lambda)$ is clearly minimal. It follows $\Lambda=\ring_D$.
\end{proof}

\subsection{Definiteness}\label{section-definitness}
In this section, let the ground field $F$ be totally real, and let $E=F(\sqrt{-d})$, where $d\in\ring_F$ is square free and totally positive, be an imaginary quadratic extension.
The hermitian form of section~\ref{section-unitary-groups} is called totally definite, if it is definite at all the archimedean places, equivalently, if the unitary group $\G(F_v)$
is compact for all $v\mid\infty$. 
At an archimedean place $v\mid\infty$ we have $F_v\cong \R$ and $E_v\cong \C$. So $L_v\cong \C\oplus\C\oplus\C$, where the embedding $L\hookrightarrow\C\oplus\C\oplus\C$ is given by
the three embeddings of $L$ to $E_v$, which we again denote by $\rho^j$, $j=0,1,2$, $l\mapsto(\rho^0(l),\rho(l),\rho^2(l))$.
Using the cyclic presentation, 
$\D(F_v)$ is isomorphic to $M_3(\C)$ equipped with the involution~(\ref{matrix-invol}). But on $M_3(\C)$ we have the obvious involution $\beta$ of the second kind given by conjugate transpose,
$\beta(M)=\bar M'$, and $M_3(\C)$ being central simple, there exists $H\in\GL_3(\C)$ such that $\alpha(M)=H^{-1}\beta(M)H^{-1}$. Obviously, this is satisfied by
\begin{equation*}
  H\:=\:\begin{pmatrix}\rho^0(b)&&\\&\rho(b)&\\&&\rho^2(b)\end{pmatrix}\:,
\end{equation*}
and $M\in M_3(\C)$ belongs to $\U(F_v)$ if and only if
\begin{equation*}
 \bar M'HM\:=\:H\:.
\end{equation*}
Thus, $\U(F_v)$ is isomorphic to the  unitary group associated to the standard hermitian form  given by  $H$. This is definite if and only if $\rho^0(b)$, $\rho(b)$ and $\rho^2(b)$ 
have the same sign. 
 In view of $\n_{M/F}(b)=\n_{E/F}(a)$ being positive for any embedding $L\hookrightarrow\C$, we actually have $\epsilon_v=1$ for all $v$. That is, $b$ must be totally positive.
We have found a simple criterion for definiteness.
\begin{prop}\label{prop-definite}
Assume the division algebra of Theorem~\ref{thm-involutions} is defined over an 
imaginary quadratic extension $E/F$ of a totally real number field $F$. Then the involution defined by the totally real number $b$ is totally definite if and only if
 $b$ is totally positive.
\end{prop}

\section{Results on rational points of the unitary group}
\subsection{Fixed points}\label{section-fixed-points}
In the situation of Theorem~\ref{thm-involutions} we assume  the field $E$ to equal $F(\sqrt{-d})$, where $-d$ is a square-free element of $\ring_F$.
By the embedding (\ref{embedding}), an element of $\D(F)$ is given by the first row $(l_0,l_1,l_2)$ of the corresponding matrix, $l_0,l_1,l_2\in L$.
The unitary condition $g\alpha(g)=1$ with respect to $\alpha$ given by (\ref{matrix-involution}) for such an element is given by 
\begin{equation}\label{unit1b}
 l_0 \bar l_0+\rho(b)l_1\bar l_1+\rho(b)\rho^2(b)l_2\bar l_2\:=\:1
\end{equation}
together with
\begin{equation}\label{unit2b}
a \bar l_0 \rho(l_2)+\rho(b)\bar l_1 \rho(l_0)+\rho(b)\rho^2(b)\bar l_2 \rho(l_1)=0.
\end{equation}
The determinant condition is
\begin{equation}\label{det}
 N_{L/E}(l_0)+aN_{L/E}(l_1)+a^2N_{L/E}(l_2)-a\trace\nolimits_{L/E}(l_0\rho(l_1)\rho^2(l_2))=1.
\end{equation}
So an element $g=g(l_0,l_1,l_2)$ satisfying (\ref{unit1b}) and (\ref{unit2b}) belongs to the unitary group $U=\U(F)$ defined by the involution $\alpha$ on the division algebra. 
If it additionally satisfies (\ref{det}), then it belongs to the special unitary group $\mathop{SU}=\G(F)$.
We have an action of $\mathop{Gal}(L/E)=<\rho>$ on the cyclic algebra $\D(F)=L\oplus Lz\oplus Lz^2$ by automorphisms given by the action on the coefficients,
\begin{equation*}
 \rho(l_0+l_1z+l_2z^2)\:=\:\rho(l_0)+\rho(l_1)z+\rho(l_2)z^2\:,
\end{equation*}
for all $l_j\in L$. This coincides with the inner automorphism given by conjugation with $z^j$, $\rho^j(d)=z^j d z^{-j}$. The  $\mathop{Gal}(L/E)$-fixed points of $\D(F)$
obviously are those with coefficients  $l_j\in E$.

Let us collect some simple properties.
\begin{prop}\label{monomial_elements}
\begin{itemize}
 \item [(a)]
 If for an element $g(l_0,l_1,l_2)\in\U(F)$ one of the coefficients is zero, then $g=lz^j$, $l\in L$, is monomial.
 \item [(b)]
 The monomial elements $g=lz^j$, $j=0,1,2$, of $\U(F)$ are given by those $l\in L$ satisfying the norm equation 
 $\n_{L/M}(l)=1$, $\rho(b)\n_{L/M}(l)=1$, $\rho(b)\rho^2(b)\n_{L/M}(l)=1$, respectively, for $j=0$, $j=1$, $j=2$, respectively.
 \item [(c)]
 The monomial elements $g$  of $\G(F)$ are the elements given by $g=l$, where $l\in L^\times$ satisfies the norm equations
\begin{equation}\label{norm-cond-monomial}
 \n_{L/M}(l)\:=\:1\:=\:\n_{L/E}(l)\:.
\end{equation}
Especially, the $\mathop{Gal}(L/E)$-fixed monomial elements are given by the third roots of unity contained in $E$.
\end{itemize}
\end{prop}
\begin{proof}[Proof of Proposition~\ref{monomial_elements}]
 If $l_k=0$, then by the unitary condition (\ref{unit2b}) a second coefficient $l_n$, $n\not= k$, is zero, too. So $g=l_jz^j$ for $j\not=k,n$. This is (a).

 Concerning (b), for a monomial element $g=lz^j$ the unitary condition (\ref{unit1b}) clearly simplifies to the stated ones.
 Concerning (c), $g=lz^j\in\G(F)$ must satisfy the determinant condition (\ref{det}), $a^j\n_{L/E}(l)=1$.
  As $a$ and $a^2$ don't belong to $\n_{L/E}$, we must have $j=0$, and hence (\ref{norm-cond-monomial}).
In the special case $l\in E^\times$, we have $l\bar l=1=l^3$, i.e. $l$ is a third root of unity. The primitive third roots of unity belong to $E=F(\sqrt{-d})$ if and only if $d=3$, 
i.e. $E=F(\zeta_6)$.
\end{proof}
 By Hilbert~90
  an element $l\in L^\times$  satisfying the first condition of (\ref{norm-cond-monomial}) in Proposition~\ref{monomial_elements} is of the form
 \begin{equation*}
  l\:=\:\frac{y_0+\sqrt{-d}y_1}{y_0-\sqrt{-d}y_1}\:.
 \end{equation*}
 with $y_0,y_1\in M$.
In order to satisfy the second condition non-trivially, we may assume $y_1$ to be non-zero and normalize it $y_1=1$.
Then the second condition is equivalent to $\trace_{L/E}(y_0\rho(y_0))=d$.

Although conjugation of the coefficients $l_j\in L$, $l_0+l_1z+l_2z^2\mapsto \bar l_0+\bar l_1z+\bar l_2z^2$, does not define an algebra homomorphism, 
we can ask for elements of $\U(F)$ and $\G(F)$ whose coefficients
are fixed under conjugation, respectively mapped to their negative, i.e. $l_j\in M$ for all $j$, respectively $l_j\in\sqrt{-d}M$ for all $j$. 

\begin{prop}\label{M-points}
\begin{itemize}
 \item [(a)] 
 The only element of $\G(F)$ given by a first row $(l_0,l_1,l_2)$ such that $\bar l_j=\epsilon l_j$, $j=0,1,2$, with $\epsilon\in\{\pm1\}$, is the identity.
 \item [(b)]
 The elements of $\U(F)$ given by a first row $(l_0,l_1,l_2)$ such that $\bar l_j= l_j$, $j=0,1,2$, are $g=\pm 1$, $g=lz$ if  $l\in M$ is a solution of $\rho(b)l^2=1$, and 
 $g=lz^2$ if $l\in M$ is a solution of $\rho(b)\rho^2(b)l^2=1$.
 \item [(c)]
 The elements of $\U(F)$ given by a first row $(l_0,l_1,l_2)$ such that $\bar l_j= -l_j$, $j=0,1,2$, are $g=l$ if $l\in\sqrt{-d}M$ is a forth root of unity, 
 $g=\sqrt{-d}\cdot lz$ if $l\in M$ is a solution of $d\rho(b)l^2=1$, and 
 $g=\sqrt{-d}\cdot lz^2$ if $l\in M$ is a solution of $d\rho(b)\rho^2(b)l^2=1$.
\end{itemize}
 \end{prop}
\begin{proof} Let $\epsilon\in\{\pm1\}$.
 Let $(l_0,l_1,l_2)$ be the first row of an element $g$ of $\U(F)$ satisfying $\bar l_j=\epsilon l_j$ for $j=0,1,2$.
 As $a\in E\setminus F$, condition (\ref{unit2b}) splits into two conditions
 \begin{equation*}
  l_0\rho(l_2)=0\quad \textrm{ and }\quad l_1\rho(l_0)+\rho^2(b)l_2\rho(l_1)=0.
 \end{equation*}
By the first one, $l_0$ or $l_2$ is zero. Then, by the second condition the other one or $l_1$ is zero, too.
So $g=lz^k$ is monomial.
The proposition now follows easily by evaluation of condition (\ref{unit1b}).
\end{proof}
The involution $\alpha$ is an extension of the conjugation on $L$ to an algebra anti-automorphism. For this we have the following theorem, 
which has Corollaries~\ref{no-elements-of-even-order} and \ref{reflections-non-division} as  surprising consequences.
\begin{thm}\label{involution-eigenspaces}
There is no element of $\G(F)$ apart from the identity which is an eigenvector for the involution $\alpha$, i.e. $\alpha(g)\not=\pm g$ for all $g\in\G(F)$.
The eigenvectors of $\U(F)$ under $\alpha$ are the forth roots of unity in $E^\times$.
\end{thm}
\begin{cor}\label{no-elements-of-even-order}
 The group $\G(F)$ does not contain non-trivial elements of even order.
 Especially, there are no reflections in $\G(F)$.
\end{cor}
\begin{proof}[Proof of Corollary~\ref{no-elements-of-even-order}]
 If there was an element of finite even order, then there would exist an element $g\in \G(F)$ of order two as well. 
 But then $g=g^{-1}=\alpha(g)$ would belong to the $+1$-eigenspace of the involution $\alpha$, in contradiction to Theorem~\ref{involution-eigenspaces}.
\end{proof}
\begin{cor}\label{reflections-non-division}
 Let $H$ be a hermitian matrix in $M_3(\C)$, and let $\SU_3(H)=\{g\in\SL_3(\C)\mid \bar g'Hg=H\}$ denote the associated special unitary group.
 Let $R\in\SU_3(H)$ be a reflection. Then $R$ is not contained in any  global division algebra
 $(D,\alpha)$ of degree three with involution $\alpha$ of the second kind  such that for some infinite place $v$,  $\SU(D_v)\cong \SU_3(H)$.
 Similarly, no hermitian or skew-hermitian element of $\SU_3(H)$ arises in this way.
\end{cor}

For the proof of Theorem~\ref{involution-eigenspaces} we need the following lemma.
\begin{lem}\label{algebraic-lemma}
If $a$ satisfies the conditions of Theorem~\ref{thm-involutions}, then it holds $aF\cap N_{L/E}=\{0\}$.
\end{lem}
\begin{proof}[Proof of Lemma~\ref{algebraic-lemma}]
 First we notice that if $q\in F\cap N_{L/E}$, then there exists $x\in M$ such that $N_{L/E}(x)=q$.
 As, let $y\in L$ be such that $N_{L/E}(y)=q$. Let $z=N_{L/M}(y)$. By the compatibility of norms we have $q^2=N_{L/M}(N_{L/E}(y))=N_{M/F}(z)$. 
 This implies $N_{L/E}(\frac{z^2}{q})=N_{M/F}(\frac{z^2}{q})=\frac{q^4}{q^3}=q$.
 Now assume that for some $q\in F^\times$ we have $aq\in N_{L/E}$. It follows that $N_{L/M}(aq)=a\tau(a)q^2$ belongs to $N_{M/F}$. So by an argument analog to the above, $q\in N_{M/F}$. 
 But this implies $a\in N_{L/E}$, a contradiction.
\end{proof}
\begin{proof}[Proof of Theorem~\ref{involution-eigenspaces}]
 For an element $g\in\mathbb D(F)$ given by a first row $(l_0,l_1,l_2)$, the condition $\alpha(g)=\pm g$ reduces to $l_0=\pm\bar l_0$ and $l_2=\pm \frac{b}{a}\rho^2(\bar l_1)$.
 Now assuming $g$ to belong to $\U(F)$, the unitary condition (\ref{unit2b}) multiplied by $l_1$ gives the following equation.
 \begin{equation*}
  \rho(b)l_1\bar l_1(l_0+\rho(l_0))=\mp al_1\rho(l_1)\rho^2(l_1).
 \end{equation*}
The right hand side is $\mp aN_{L/E}(l_1)$, which belongs to $E$. While the conjugate of the left hand side is $\pm \rho(b)l_1\bar l_1(l_0+\rho(l_0))$, 
so belongs either to $M$ or to $\sqrt{-d}M$.
It follows that $aN_{L/E}(l_1)\in E \cap M=F$, respectively $aN_{L/E}(l_1)\in E\cap \sqrt{-d}M=\sqrt{-d}F$. 
Both cases give the condition $N_{L/E}(\bar l_1')\in aF$ (where $l_1'=l_1$ in the first case, and $l_1'=l_1/\sqrt{-d}$ in the second case). 
By Lemma~\ref{algebraic-lemma},  $l_1=0$ hence $l_2=0$, and  $g=l_0$ must satisfy $l_0^2=\pm 1$, so $l_0$ is a forth root of unity in $L$. As $L/E$ does not admit a subextension of degree two,
$l_0$ already belongs to $E$. For $l_0\in\G(F)$ we must have $l_0^3=1$, so $l_0=1$.
\end{proof}


\subsection{Integer valued points}\label{section-integers}
We assume $F$ to be totally real and $E/F$ to be  imaginary quadratic.
In working with integer valued points we have to take into account the discussion of their definition in section~\ref{section-unitary-groups}.
But even in the case of $a\in E^\times$ or $b\in M^\times$ not being units in the corresponding rings of integers, to ask for  possible denominators
of the coefficients of the $F$-valued points in the cyclic presentation  is interesting.

We give an answer of this question in two special cases. 
First, for the case of monomial elements. Second, notice that in case the quantity $b$ defining the involution $\alpha$ can be chosen to belong to $F$, 
the subgroups $\U(F)^\rho$ and $\G(F)^\rho$ of $\mathop{Gal}(E/F)$-fixed points in $\U(F)$ and $\G(F)$, respectively, give themselves rise to group schemes over $F$.
They are associated to the unitary, respectively, special unitary group for the hermitian form on $E^3$ induced by the involution of (\ref{matrix-invol}) restricted to the subalgebra $M_3(E)$ of $M_3(L)$.
We characterize the denominators of $\U(F)^\rho$ and $\G(F)^\rho$ in this case.

\begin{defn}
 For a set $S$ of prime ideals of the number field $F$ we denote by $\ring_F(S)$ the subring of $F$ in which exactly the prime ideals in $S$ are invertible.
\end{defn}
\begin{defn}
\begin{itemize}
 \item [(a)] We say a  prime ideal $\pp$ of $F$ satisfies \textbf{Property A} for the extension $L/E/F$, if $\pp$ does not contain two, and if $\pp$ is inert in $E$ but splits in $L$.
 \item [(b)] 
 We say a prime ideal $\pp$ of $F$ satisfies \textbf{Property B} for the extension $E$, if $\pp$ does not contain two, and $\pp$ is either inert or ramified in $E$ such that 
 ${F_\pp}$ does not contain the sixth roots of unity.
\end{itemize}
\end{defn}

\begin{prop}\label{prop-S-monomials}
 Let $S$ be a set of primes $\pp$ of $F$ satisfying Property $A$, and for which the valuations $v_\pp(b)$ are zero. 
 Then the monomial elements $g=lz^j\in\G(F)$ with $l\in\ring_L(S)$ are the third roots of unity contained in $E$.
\end{prop}
Here  $\ring_L(S)$ denotes $\ring_L(S)\:=\:\ring_E(S)\lambda_0+\ring_E(S)\lambda_1+\ring_E(S)\lambda_2$
for any  integral basis $\lambda_0,\lambda_1,\lambda_2$ of $\ring_L$.
\begin{proof}[Proof of Proposition~\ref{prop-S-monomials}]
 For $\pp\in S$ let $\pp_1,\pp_2,\pp_3$ be the prime ideals of $\ring_L$ above $\pp$, 
 so $\pp\ring_L=\pp_1\pp_2\pp_3\ring_L$. 
 Recall the field $M$ is the subfield of $L$ fixed by conjugation.
 As $\pp$ is unramified in $E$, the prime ideals $\pp_k$ are defined by their intersection with $M$, $\pp_k=(\pp_k\cap\ring_M)$. 
 For $l\in\ring_L(S)$, there are integers $r_k$ such that $v_{\pp_k}(l)=r_k$, $k=1,2,3$.
  Let $g=lz^j$ satisfy the unitary condition (\ref{unit1b}), 
 $1=c(j,b)\n_{L/M}(l)$ where $c(0,b)=1$, $c(1,b)=\rho(b)$, $c(2,b)=\rho(b)\rho^2(b)$ have $\pp_k$-valuation zero. But then $v_{\pp_k}(\n_{L/M}(l))=2r_k$ must be zero, so $r_1=r_2=r_3=0$.
Varying $\pp$ in $S$ it follows that $l\in\ring_L^\times$.
As $M/F$ is totally real, the unit group $\ring_M^\times$ is isomorphic to $\Z/2\Z\oplus\Z^r$ for some $r\geq 2$. Let $e_1,\dots,e_r$ be generators  of the non-torsion part.
Then $\ring_L^\times=\ring_E^\times\times <e_1,\dots,e_r>$. Accordingly, write $l=\xi e_1^{s_1}\cdots e_r^{s_r}$. Then the unitary condition is 
\begin{equation*}
 1\:=\:c(j,b)\xi\bar\xi  \cdot e_1^{2s_1}\cdots e_r^{2s_r}\:.
\end{equation*}
Especially, when $j=0$ this is satisfied only if $s_1=\dots=s_r=0$ and $1=\xi\bar\xi$. 
In case $g=l\in\G(F)$ the determinant condition $1=\n_{L/E}(\xi)=\xi^3$ implies $\xi$ is a third root of unity.
\end{proof}

\begin{thm}\label{thm-no-o_E(S)-points}
 For the division algebra of Theorem~\ref{thm-involutions} assume that for the structure constant $a\in\ring_E$ the involution is given
by some $b\in F$.
Let $S$ be a set of prime ideals $\pp$ of $F$ satisfying Property B and for which the valuations $v_\pp(b)$ are zero.
 Then, apart from the monomial solutions $g=lz^j$, where $l\in \ring_E^\times$ with $b^jl\bar l=1$,  
 there is no element in $\U(F)$ given by coordinates $(l_0,l_1,l_2)$ with $l_j\in\ring_E(S)$ 
 for $j=0,1,2$ in the cyclic presentation.
 Especially, the elements of $\G(F)^\rho$ of this kind are   the third roots of unity contained in $E$.
\end{thm}
In the case of the coincidence of the maximal order of integer points $\ring_D$ with the cyclic order $\Lambda=\ring_L\oplus\ring_L z\oplus \ring_Lz^2$, Theorem~\ref{thm-no-o_E(S)-points}
implies the triviality of the $\ring_F(S)$-valued points of $\G^\rho$. We formulate this in the Kummer case, i.e. in case $E=F(\zeta_3)$.
\begin{cor}\label{cor-integral-case}
For the division algebra of Theorem~\ref{thm-involutions} assume  $E=F(\zeta_3)$, and assume the constants are $a\in\ring_E^\times$ and $b\in\ring_F^\times$.
Then for all sets $S$ of prime ideals of $F$ ramified or unramified and non-split in $E$ the group $\G(\ring_F(S))^\rho$ of $\ring_F(S)$-valued points is $\{1,\zeta_3,\bar\zeta_3\}$.
\end{cor}
For the proof of Theorem~\ref{thm-no-o_E(S)-points} we need the following lemmas.
\begin{lem}\label{homogenous_solutions_modp}
Assume for a prime $p$ and an integer $n$ that $p^{n}\equiv 5\mod 6$.
Let $a\in \F_{p^{2n}}\setminus\F_{p^n}$ satisfy $a\bar a=b^3$ for some $b\in\F_{p^{n}}^\times$.
 Then the following system of homogeneous equations 
 \begin{equation}\label{Unit1modp}
l_0\bar l_0+bl_1\bar l_1+b^2 l_2\bar l_2 =0, 
\end{equation}
\begin{equation}\label{Unit2modp}
 a\bar l_0l_2+b\bar l_1l_0+b^2\bar l_2l_1=0,
\end{equation}
for $l_j\in \F_{p^{2n}}$, $j=0,1,2$,
only has the trivial solution $(l_0,l_1,l_2)=(0,0,0)$.
\end{lem}
\begin{proof}[Proof of Lemma~\ref{homogenous_solutions_modp}]
First notice that the condition $p^{n}\equiv 5\mod 6$ is satisfied only for primes $p\equiv 5\mod 6 $ and odd $n$. 
Equivalently, there exists no primitive sixth root of unity in the finite field $\F_{p^n}$.
 Next, if one of the $l_j$ is zero, then by (\ref{Unit2modp}) another one is zero. But then by (\ref{Unit1modp}), the remaining one must be zero, too.
 So for a non-trivial solution we have $l_0l_1l_2\not=0$. As the equations are homogeneous, we may assume  without loss of generality $l_2=1$.
 Then equations (\ref{Unit1modp}) and (\ref{Unit2modp}) simplify to
 \begin{equation}\label{Unit1modp1}
  l_0\bar l_0+bl_1\bar l_1+ b^2 =0,
 \end{equation}
\begin{equation}\label{Unit2modp1}
 a\bar l_0+b\bar l_1l_0+b^2l_1=0.
\end{equation}
From (\ref{Unit2modp1}) and its conjugate we get the following system of linear equations for $l_1,\bar l_1$,
\begin{equation*}
 \begin{pmatrix}
  b^2&bl_0\\b\bar l_0&b^2
 \end{pmatrix}
\begin{pmatrix}
  l_1\\\bar l_1 
 \end{pmatrix}
=\begin{pmatrix}
  -a\bar l_0\\-\bar a l_0 
 \end{pmatrix}.
\end{equation*}
Multiplying by the adjunct 
$\begin{pmatrix}
  b^2&-bl_0\\-b\bar l_0&b^2
 \end{pmatrix}$
of the matrix involved we get
\begin{equation*}
 b^2(b^2-l_0\bar l_0)\begin{pmatrix}
                 l_1\\\bar l_1 
                \end{pmatrix}
=\begin{pmatrix}
  -ab^2\bar l_0+\bar abl_0^2\\ab\bar l_0^2-\bar ab^2l_0
 \end{pmatrix}.
\end{equation*}
There are two possibilities. First, assume $b^2-l_0\bar l_0\not=0$.
Then the linear equation has the solution $l_1=\frac{\bar abl_0^2-ab^2\bar l_0}{b^2(b^2-l_0\bar l_0)}$.
Inserting
\begin{equation*}
 l_1\bar l_1 =\frac{a\bar ab^2(l_0\bar l_0)^2-a^2b^3\bar l_0^3-\bar a^2b^3l_0^3+a\bar ab^4l_0\bar l_0}{b^4(b^2-l_0\bar l_0)^2}
\end{equation*}
into (\ref{Unit1modp1}) yields
\begin{equation*}
 0=(b^3-a\overline{ a^{-1} l_0^3})(b^3-\bar a a^{-1} l_0^3),
\end{equation*}
which is equivalent to $l_0^3=a^2$ . 
It follows $(l_0\bar l_0)^3=(a\bar a)^2=b^6$, which is equivalent to $l_0\bar l_0=b^2$ as there is no primitive third root of unity in $\F_{p^n}^\times$. 
This is a contradiction to the assumption $b^2-l_0\bar l_0\not=0$.

Second, assume $l_0\bar l_0=b^2$, i.e. $\bar l_0=b^2l_0^{-1}$. In this case the linear system above forces $\bar abl_0^2=ab^2\bar l_0$, i.e. $l_0^3=a^2$. 
Inserting $l_0$ into (\ref{Unit1modp1}) yields $l_1\bar l_1=-2b$, and multiplying the original equation (\ref{Unit2modp1}) by $\bar al_0$ yields
\begin{equation*}
 a\bar l_0\bar l_1+\bar al_0l_1=-b^3.
\end{equation*}
These two equations imply
\begin{equation*}
 (l_0l_1-a)(\overline{l_0l_1}-\bar a)=(l_0\bar l_0)(l_1\bar l_1)-(a\overline{l_0l_1}+ \bar al_0l_1)+a\bar a=-2b^3+b^3+b^3=0.
\end{equation*}
Equivalently, $l_0l_1=a$.  But if so, $l_1\bar l_1=\frac{a\bar a}{l_0\bar l_0}=b$ in contradiction to $l_1\bar l_1=-2b$.
So in neither case we get a non-trivial solution satisfying the conditions of Proposition~\ref{homogenous_solutions_modp}.
\end{proof}
By a rather  similar argument we get the following.
\begin{lem}\label{homogenous_solutions_mod3}
Let $p$ be an odd prime and such that for some $n$ the finite field $\F_{p^n}$ does not contain the third roots of unity.
On the residue class ring $R:=\F_{p^n}[\pi]/(\pi^2)$ let $\:\bar{}\:$ denote the conjugation given by $\bar\pi=-\pi$.
 Let $a\in R^\times\setminus\F_{p^{n}}^\times$ and $b\in\F_{p^{n}}^\times$  satisfy the relation $a\bar a=b^3$.
 Then the following system of  equations  
  \begin{equation}\label{Unit1mod3}
l_0\bar l_0+bl_1\bar l_1+b^2 l_2\bar l_2 =0, 
\end{equation}
\begin{equation}\label{Unit2mod3}
 a\bar l_0l_2+b\bar l_1l_0+b^2\bar l_2l_1=0,
\end{equation}
for $l_0,l_1,l_2\in R $,
only has the solutions $(l_0,l_1,l_2)\equiv (0,0,0) \mod \pi$.
\end{lem}
\begin{proof}[Proof of Lemma~\ref{homogenous_solutions_mod3}]
If one of the $l_j$ is zero modulo $\pi$, then by (\ref{Unit2mod3}) another one is zero modulo $\pi$. But then by (\ref{Unit1mod3}), the remaining one must be zero modulo $\pi$, too.
 So for a solution not contained in $\pi R$ we have $l_0l_1l_2\not\equiv 0 \mod \pi$. As the equations are homogeneous, we may assume  without loss of generality $l_2=1$.
 Then equations (\ref{Unit1modp}) and (\ref{Unit2modp}) simplify to
 \begin{equation}\label{Unit1mod31}
  l_0\bar l_0+bl_1\bar l_1+ b^2 =0,
 \end{equation}
\begin{equation}\label{Unit2mod31}
 a\bar l_0+b\bar l_1l_0+b^2l_1=0.
\end{equation}
From (\ref{Unit2modp1}) and its conjugate we get the following system of linear equations for $l_1,\bar l_1$,
\begin{equation*}
 \begin{pmatrix}
  b^2&bl_0\\b\bar l_0&b^2
 \end{pmatrix}
\begin{pmatrix}
  l_1\\\bar l_1 
 \end{pmatrix}
=\begin{pmatrix}
  -a\bar l_0\\-\bar a l_0 
 \end{pmatrix}.
\end{equation*}
Multiplying by the adjunct 
$\begin{pmatrix}
  b^2&-bl_0\\-b\bar l_0&b^2
 \end{pmatrix}$
of the matrix involved we get
\begin{equation*}
 b^2(b^2-l_0\bar l_0)\begin{pmatrix}
                 l_1\\\bar l_1 
                \end{pmatrix}
=\begin{pmatrix}
  -ab^2\bar l_0+\bar abl_0^2\\ab\bar l_0^2-\bar ab^2l_0
 \end{pmatrix}.
\end{equation*}
There are two possibilities. First, assume $b^2-l_0\bar l_0\not\equiv 0 \mod \pi$.
Then the linear equation has the solution $l_1=\frac{\bar abl_0^2-ab^2\bar l_0}{b^2(b^2-l_0\bar l_0)}$.
Inserting
\begin{equation*}
 l_1\bar l_1 =\frac{a\bar ab^2(l_0\bar l_0)^2-a^2b^3\bar l_0^3-\bar a^2b^3l_0^3+a\bar ab^4l_0\bar l_0}{b^4(b^2-l_0\bar l_0)^2}
\end{equation*}
into (\ref{Unit1mod31}) yields
\begin{equation*}
 0=(b^3-a\overline{ a^{-1} l_0^3})(b^3-\bar a a^{-1} l_0^3)\:.
\end{equation*}
From this we get $l_0^3\equiv a^2 \mod \pi$. It follows $(l_0\bar l_0)^3\equiv (a\bar a)^2\equiv b^6 \mod \pi$. But $\F_{p^n}$ does not contain non-trivial third roots of unity, 
so $l_0\bar l_0\equiv b^2 \mod \pi$, a contradiction.

Second, assume $l_0\bar l_0\equiv b^2 \mod \pi$. In this case the linear system above forces $\bar abl_0^2\equiv ab^2\bar l_0 \mod \pi$, i.e. $l_0\equiv b \mod \pi$. 
Inserting $l_0$ into (\ref{Unit1mod31}) yields $l_1\bar l_1\equiv -2b \mod\pi$, 
and multiplying the original equation (\ref{Unit2mod31}) by $\bar al_0$ yields
\begin{equation*}
 a\bar l_0\bar l_1+\bar al_0l_1=-b^3.
\end{equation*}
These two equations imply
\begin{equation*}
 (l_0l_1-a)(\overline{l_0l_1}-\bar a)=(l_0\bar l_0)(l_1\bar l_1)-(a\overline{l_0l_1}+ \bar al_0l_1)+a\bar a\equiv -2b^3+b^3+b^3\equiv 0\mod\pi.
\end{equation*}
So $l_0l_1\equiv a\mod\pi$, and $l_0l_1\bar l_0\bar l_1\equiv a\bar a\mod\pi$. But this implies $l_1\bar l_1\equiv b\mod\pi$, in contradiction to $l_1\bar l_1\equiv -2b \mod\pi$.
So in neither case we get a  solution non-trivial modulo $\pi$ satisfying the conditions of Proposition~\ref{homogenous_solutions_mod3}.
\end{proof}
\begin{proof}[Proof of Theorem~\ref{thm-no-o_E(S)-points}]
We notice that Property B for a prime ideal $\pp$  is equivalent to the condition that the \'etale extension $E_\pp/F_\pp$ is 
a field extension, and 
such that the residue class field $\kappa_{F_\pp}$ does not  contain a primitive sixth root of unity. So $\kappa_{F_\pp}\cong \F_{p^{n}}$, and either $p=3$ or  $p^n\equiv 5\mod 6$.

  Assume $(l_0,l_1,l_2)$ gives rise to an element $g$ of $\U(F)$ with $l_j\in \ring_E(S)$, $j=0,1,2$. 
  If actually each $l_j\in \ring_E$,  then all the summands of condition (\ref{unit1b}) are non-negative.
  For a suitably chosen embedding of $F$ into $\RR$, we have $l_0\bar l_0=0$ or $l_0\bar l_0=1$.
  So either $l_0=0$ or $l_1=l_2=0$. By (\ref{unit2b}), in the first case one of $l_1,l_2$ is zero, too.
  So in order to satisfy condition (\ref{unit1b}), $g$ must be monomial, $g=lz^j$ for an element $l\in\ring_E^\times$ of norm one. 
  Applying Proposition~\ref{monomial_elements}, we only get the trivial solutions $(l_0,l_1,l_2)=(\zeta_3^{k},0,0)$ in $\G(F)$.
  
  For a non-trivial solution with $l_j\in\ring_E(S)$, $j=0,1,2$, and $l_j\notin\ring_E$ for at least one $j$,
  there exists a prime ideal $\pp\in S$ and an integer $r>0$ such that $\pp^rl_j$ belongs to the ring $\ring_{E_\pp}$ of
  $\pp$-adic integers, $j=0,1,2$. We assume $r$ to be chosen minimal with this Property.
  Let $\pi$ be a uniformizing element of the prime ideal in $\ring_{E_\pp}$.
 We get a tuple $(l_0',l_1',l_2')=\pi^r(l_0,l_1,l_2)$ which satisfies the two unitary conditions (\ref{unit1b}), (\ref{unit2b}) for $(l_0,l_1,l_2)$ 
 and leads to the two homogeneous conditions (\ref{Unit1modp}) and (\ref{Unit2modp}) 
 of Lemma~\ref{homogenous_solutions_modp} for $(l_0',l_1',l_2')$ modulo $(\pi)$, in case $E_\pp/F_\pp$ is unramified.
 Respectively, if $E_\pp/F_\pp$ is ramified, $(l_0',l_1',l_2')$ modulo $(\pi^2)$ satisfy the conditions (\ref{Unit1mod3}) and (\ref{Unit2mod3}) of Lemma~\ref{homogenous_solutions_mod3}.
 As $b\in\ring_{F_\pp}^\times$ by assumption,  Lemmas~\ref{homogenous_solutions_modp} and \ref{homogenous_solutions_mod3} apply.
 So $(l_0',l_1',l_2')$ must be zero modulo $(\pi)$.  This contradicts the minimal choice of $r$. 
\end{proof}

The following restriction for the denominators of $\U(F)^\rho$ is a consequence of the Proof of Theorem~\ref{thm-no-o_E(S)-points}.
\begin{cor}\label{denominators}
 For the division algebra of Theorem~\ref{thm-involutions} assume   $b\in F$.
 Let $g(l_0,l_1,l_2)$ be an element of $\U(F)^\rho$. 
 Then the denominators of $l_j\in E$ are not contained in prime ideals $\pp$ of $E$ lying over prime ideals of $F$ satisfying Property B such that $v_\pp(b)=0$.
\end{cor}
\begin{proof}[Proof of Corollary~\ref{denominators}]
Let $\pp$ be a prime satisfying the conditions above and consider $g(l_0,l_1,l_2)$ as an element of $\U(F_\pp)^\rho$.
If $g(l_0,l_1,l_2)\notin\U(\ring_{F_\pp})^\rho$, then there is an integer $r>0$ (again chosen minimally) such that $\pi^r(l_0,l_1,l_2)=(l_0',l_1',l_2')$ 
satisfies (\ref{Unit1modp}) and (\ref{Unit2modp}) modulo $(\pi)$ (respectively, (\ref{Unit1mod3}) and (\ref{Unit2mod3}) modulo $(\pi^2)$), 
and by Lemma~\ref{homogenous_solutions_modp} (respectively, Lemma~ \ref{homogenous_solutions_mod3}),  $(l_0',l_1',l_2')\equiv (0,0,0)$ modulo $\pi$, which contracts the minimal choice of $r$.
\end{proof}

The notion of $\ring_F(S)$-valued points becomes relevant when there is a $\ring_F$-structure on the special unitary group $\G$.
Then the group $\G(\ring_F(S))$ is an arithmetic subgroup of $\G(F_\pp)$. 
Especially, if $\G(F_v)$ is compact for some archimedean place $v$, the quotient will be cocompact (\cite{borel}, \cite{borel-harder}).
Any explicit description of $\G(\ring_F(S))$, or of some congruence subgroup 
allows to deduce properties of the quotient $\G(\ring_F(S))\backslash \G(F)$. 
But the following example gives evidence that by choosing a cyclic presentation, i.e.  controlling  the involution $\alpha$, 
even the explicit detection of non-trivial elements of $\G(\ring_F(S))$ is sophisticated.
\begin{example}\label{example-S-arithmetic}
 For simplicity we assume the ground field $F$ to equal the rationals $\Q$, but this example  generalizes to arbitrary totally real ground fields, when  the elements 
 of finite order in $D$ can be controlled.
 Let $E=\Q(\sqrt{-3})=\Q(\zeta_3)$, where $\zeta_3=\frac{1}{2}(-1+\sqrt{-3})$. Let $M/\Q$ be a totally real $C_3$-Galois extension such that $L=EM$ be $C_6$-Galois over $\Q$. 
 We assume that the norm of ${L/E}$ is  not surjective on $\ring_E^\times$ (that is, $L$ does not contain the ninth roots of unity).
 For example, one could generate $M$ by the polynomial $f(X)=X^3-13X+13$.
 Choose the structure constant $a=\zeta_3^j$ of the division algebra $D$ in Theorem~\ref{thm-involutions} to be a power of $\zeta_3$,  where $j\not\equiv 0\mod 3$. Then $a\notin\n_{L/M}$
 and $D$ is indeed a division algebra. As $a\bar a=1$, we choose the involution constant $b=1$.
 So the involution  (\ref{matrix-involution}) written with respect to the cyclic presentation is
 \begin{equation*}
 \alpha\::\: 
 \begin{pmatrix}  l_0&l_1&l_2\\a\rho(l_2)&\rho(l_0)&\rho(l_1)\\a\rho^2(l_1)&a\rho^2(l_2)&\rho^2(l_0)\end{pmatrix}\:
 \mapsto\: 
 \begin{pmatrix}
  \bar l_0&\bar a\rho(\bar l_2)&\bar a\rho^2(\bar l_1)\\
  \bar l_1&\rho(\bar l_0)&\bar a\rho^2(\bar l_2)\\
  \bar l_2&\rho(\bar l_1)& \rho^2(\bar l_0)
 \end{pmatrix}\:.
\end{equation*}
The unitary groups $\U$ and $\G$ are defined over $\Z$, and the different notions of integer valued points coincide.
At infinity the involution $\alpha:M_3(\C)\to M_3(\C)$ is simply given by the conjugate transpose, $\alpha(g)=\bar g'$. 
Especially, $\U(\R)$ and $\G(\R)$ are compact (see Proposition~\ref{prop-definite}).

 The subfield $E(z)$ of $D$ is isomorphic to $\Q(\zeta_9)$. Theorem~\ref{thm-no-o_E(S)-points}, Corollary~\ref{cor-integral-case} and Corollary~\ref{denominators}
 say that for elements $l_0+l_1z+l_2z^2\in E(z)$
 to belong to $\U(\Q)$ it is necessary that the primes $p$ occurring in the denominators of $l_0,l_1,l_2$ are  split in $E$.
 Especially, for a set $S$ of primes $p\equiv 5\mod 6$, the subgroup $\G(\Z(S))^\rho$  is given by the third roots of unity $1,\zeta_3,\zeta_3^{2}$ in $E$.
 Notice that $\G(\Z(S))^\rho$ is the intersection of $\G(\Z(S))$ with $E(z)^\times$
 Additionally, assume that the primes of $S$ actually satisfy Property A, that is $L_p$ is a split algebra.
 Then by Proposition~\ref{prop-S-monomials}, the intersection of $\G(\Z(S))$ with $L$ is $\{1,\zeta_3,\zeta_3^{2}\}$, too:
 
 \textit{In $S$-arithmetic subgroups ($S$ satisfying Property A), the elements belonging to the two obvious subfields $L$ and $E(z)$ of $D$ are the trivial ones contained in $\ring_E^\times$.}
 
 The meaning of Property A is the following. Let $p$ be a prime such that $E_p$ is non-split and $L_p$ is split. So $D_p$ is split, 
 and the embedding~(\ref{embedding}) identifies $D_p=D\otimes_\ZZ \Q_p$ with $M_3(E_p)$
 by the isomorphism $L_p\cong E_p\oplus E_p\oplus E_p$ given by the three embeddings $\rho^j$ of $L$ into $E_p$.
 Then $\G(\Q_p)$ is isomorphic to $\SU_3(E_p)$, the up to equivalence unique special unitary group over $E_p/F_p$ of degree three (see~\cite[1.9]{rogawski}).
 And the  isomorphism is given by the above embedding. Then for $S=\{(p)\}$ the group $\G(\Z(S))=\G(\Z[\frac{1}{p}])$ is an arithmetic subgroup of $\SU_3(E_p)$.
 As $\G(\R)$ is compact, $\G(\Z(S))$ is cocompact in $\SU_3(E_p)$  (see \cite{borel}, \cite{borel-harder}).
 For the resulting quotient it is natural to  study the action of the arithmetic subgroup on the affine Bruhat-Tits tree, the quotients becoming finite graphs. 
  In case $p$ ramified, this is the $(p+1)$-regular $\SL_3$-tree, in case $p$ is unramified, we get a $(p+1,p^3+1)$-bi-regular tree.
  In view of Lemma~\ref{lemma-torsion-freeness},  the  finite quotient graphs modulo $\G_0\cap\G(\Z[\frac{1}{p}])$ will be  Ramanujan (\cite{LPS}), or, 
  respectively, bi-Ramanujan (\cite{cristina-dan}). The latter case was treated in~\cite{wine-paper}, and this article is in some sense its conceptional continuation.
 \end{example}
 \begin{lem}\label{lemma-torsion-freeness}
  In the cyclic presentation of the division algebra let $E=\Q(\zeta_3)$, $a=\zeta_3^j$, ($j\not\equiv 0\mod 3$), and $b=1$.
  Then the special unitary group $\G(\Q)$ is the direct product of $<\zeta_3>$ with a torsion-free subgroup $\G_0$.
 \end{lem}
\begin{proof}
We first detect the possible orders of torsion points in $D$.
Assume $n\not=2,3,6$, and let $\xi$ be a primitive $n$-th root of unity in $D$. Then $\Q(\xi)/\Q$ is the $n$-th cyclotomic extension, which has degree $\phi(n)$. 
As $D$ is of degree three over $E$, we must have $\phi(n)=6$. It follows that the prime divisors of $n$ are contained in $\{2,3,7\}$. But the minimal polynomial of $\zeta_7$ is
$X^6+X^5+X^4+X^3+X^2+X+1$, which remains irreducible over $E$. So no seventh root of unity can be contained in $D$, but only the  eighteenth.

As any sixth root of unity in $D$ belongs to $E$, a torsion element of $\G(\Q)$ which is non-central must have order nine (eighteen being excluded by Corollary~\ref{no-elements-of-even-order}).
By assumption, we have chosen $z$ to be a primitive ninth root of unity in $D$, and $\n_{rd}(z)=a=\zeta_3^j\not=1$, so $z\not\in \G(F)$.
Let $y\in D$ be another primitive ninth root of unity. Then $z\mapsto y$ defines an automorphism on $D$, which is inner, $y=gzg^{-1}$ for some $g\in D^\times$. But then $\n_{rd}(y)=\n_{rd}(z)\not=1$, 
and so $y\not\in\G(F)$ as well.
\end{proof}


\enlargethispage{2cm}


\begin{thebibliography}{10}



\bibitem{albert2}
A.~A.~Albert:
\textit{Involutorial simple algebras and real Riemann matrices,}
Ann. Math., Vol 36, No. 4 (1935), 886-964.

\bibitem{cristina-dan}
C.~Ballantine, D.~Ciubotaru:
\textit{Ramanujan bigraphs associated with $\SU(3)$ over a $p$-adic field,}
Proc. am. math. soc. Vol. 39, Nr. 6 (2011), 1939-1953.

\bibitem{wine-paper}
C.~Ballantine, B.~Feigon, R.~Ganapathy, J.~Kool, K.~Maurischat and A.~Wooding:
\textit{Explicit construction of Ramanujan bigraphs,} in
Women in Numbers Europe: Research Directions in Number Theory, Association for Women in Mathematics Series, vol. 2, Springer (2015), 1-16.

\bibitem{borel}
A.~Borel:
\textit{Some finiteness properties of the ad\`ele groups over number fields,}
Publ. Math. IHES 16 (1963), 5-30.

\bibitem{borel-harder}
A.~Borel, G.~Harder:
\textit{Existence of discrete cocompact subgroups of reductive groups over local fields,}
Journal reine angew. Mathematik, Vol. 298  (1978), page 53-64.

\bibitem{LPS}
A.~Lubotzky, R.~Philips, P.~Sarnak:
\textit{Ramanujan graphs,}
Combinatorica, Vol. 8 (1988), 261-277.



\bibitem{reiner}
I.~Reiner:
\textit{Maximal orders,}
Academic Press, London (1975).

\bibitem{rogawski}
J.~Rogawski:
\textit{Automorphic representations of unitary groups in three variables,}
Princeton University Press (1990).
\end{thebibliography}
\end{document}